\documentclass{article}
\usepackage{fixltx2e}
\usepackage{amssymb}
\usepackage{amsfonts}
\usepackage{amsmath}
\usepackage[ansinew]{inputenc}
\usepackage{srcltx}
\usepackage[T1]{fontenc}
\usepackage{lmodern}
\usepackage{exscale}
\usepackage{graphicx}

\newtheorem{theorem}{Theorem}

\newtheorem{corollary}[theorem]{Corollary}

\newtheorem{definition}[theorem]{Definition}

\newtheorem{lemma}[theorem]{Lemma}

\newtheorem{remark}[theorem]{Remark}

\newenvironment{proof}[1][Proof]{\noindent\textbf{#1.} }{\ \rule{0.5em}{0.5em}}
\renewcommand\sb{\subseteq}
\newcommand\sm{\setminus}
\newcommand\mto{\mapsto}
\newcommand\lmto{\longmapsto}

\newcommand\cC{\mathcal{C}}
\newcommand\cD{\mathcal{D}}
\newcommand\cS{\mathcal{S}}
\DeclareMathOperator\sn{sn}

\DeclareMathOperator\cs{\mathfrak{s}}

\DeclareMathOperator\crr{\mathfrak{r}}

\DeclareMathOperator\sech{sech}
\begin{document}

\title{The Moments of L\'{e}vy's area using a sticky
shuffle Hopf algebra}
\author{Robin Hudson \\
Loughborough University,\\
Loughborough LE11 3TU\\
Great Britain \and Uwe Schauz and Wu Yue \\
Xi'an Jiaotong-Liverpool University,\\
Suzhou,\\
China}
\maketitle

\begin{abstract}
L\'{e}vy's stochastic area for planar Brownian motion is the difference of two iterated
integrals of second rank against its component one-dimen\-sional Brownian motions.
Such iterated integrals can be multiplied using the sticky shuffle product determined
by the underlying It\^{o} algebra of stochastic differentials. We use combinatorial
enumerations that arise from the distributive law in the corresponding Hopf algebra
structure to evaluate the moments of L\'{e}vy's area. These L\'{e}vy moments are well
known to be given essentially by the Euler numbers. This has recently been confirmed
in a novel combinatorial approach by Levin and Wildon. Our combinatorial
calculations considerably simplify their approach.


\end{abstract}

\section{Introduction}

L\'{e}vy's stochastic area is the signed area enclosed by the planar Brownian path and
its chord. It was originally defined rigorously by L\'{e}vy \cite{Levy} and now is
intensively applied in several areas of modern mathematics, such as rough path
analysis.

Let $B=\left( X,Y\right)$ be a planar Brownian motion in terms of components $X$ and
$Y$ which are independent one-dimensional Brownian motions.

\begin{definition}\label{defi1}
\emph{The L\'{e}vy area of }$B$\emph{\ over the time
interval }$[a,b)$\emph{\ is the stochastic integral }%
\begin{equation*}
\mathcal{A}_{[a,b)}=\dfrac{1}{2}\int_{a}^{b}\Bigl(\left( X-X(a)
\right) dY-\left( Y-Y( a) \right) dX\Bigr).  \label{Levyarea}
\end{equation*}
\end{definition}
In this definition the integral takes the same value whether it is regarded as of It\^{o} or
Stratonovich type, but in the remainder of this paper all stochastic integrals will be of
It\^{o} type, in contrast to \cite{LeWi} where the Stratonovich integral is used.

L\'{e}vy studied the characteristic function in \cite{Levy}, \cite{Levy2}, \cite{Levy3},
\cite{Levy4} and \cite{Levy1}. He derived the following formula:
\begin{theorem}\label{Levy Sech}
(L\'{e}vy \cite{Levy})
\begin{equation*}
\mathbb{E}\!\left[\exp \left(iz\mathcal{A}_{[a,b)}\right)\right]\,=\,\sech\!\left(%
\tfrac{1}{2}%
\left( b-a\right) z\right).  
\end{equation*}
\end{theorem}
We can expand the right-hand side of the formula in Theorem\,\ref{Levy Sech} using
the Taylor series
\begin{equation}
\sech(z)\,=\,\sum_{m=0}^{\infty }\left( -1\right) ^{m}\frac{A_{2m}}{\left(
2m\right) !}z^{2m}\ ,
\end{equation}%
where the even Euler zigzag numbers $A_{2m}$ are related to the Riemann zeta
function $\zeta $ by%
\begin{equation}
\zeta \left( 2m\right)\,=\,\frac{\pi ^{2m}}{\left( 2m\right) !}A_{2m}\,.
\end{equation}
This expansion shows that the nonvanishing moments of the L\'{e}vy area
$\mathcal{A}_{[a,b)}$ are given by
\begin{equation}\label{Levy Mom}
\mathbb{E}\!\left[ \mathcal{A}_{[a,b)}\right]^{2m}\,=\,\,\left(\frac{ b-a}{2}\right)^{\!2m}\!A_{2m}\,.
\end{equation}

In \cite{Levy1}, L\'{e}vy first showed Theorem \ref{Levy Sech} by using dyadic
approximation. His second proof is based on the skew product representation of planar
Brownian motion and depends on earlier work by Kac, Siegert, Cameron and Martin
(see \cite{Levy1}). In 1980, Yor \cite{Yor} simplified L\'{e}vy's proof by employing a
result on Bessel processes and an elementary result used by D.Williams. Shortly
afterwards, Helmes and Schwane \cite{Helmes} revisited the problem by extending the
$2$-dimensional set-up considered by  L\'{e}vy to $d\geq 2$ dimensions. They
considered the joint characteristic function of the stochastic process paired by the
$d$-dimensional Brownian motion $W$ and a certain generalized L\'{e}vy area
$\mathcal{A}_{[0,t]}^{J,x}$ given by
\begin{equation}
\mathcal{A}_{[0,t]}^{J,x}:=\int_0^t [W(s)+x(s)]
J(s) dW(s).
\end{equation}
Here, $[W(s)+x(s)]$ is viewed as a row-vector, a $1\times d$ matrix. 
For $d=2$, $x=0$ and $J=\bigl[ \begin{smallmatrix}
0&1\\ \!\!-1&0
\end{smallmatrix} \bigr]$ the process $\mathcal{A}_{[0,t]}^{J,x}/2$
coincides with L\'{e}vy stochastic area $\mathcal{A}_{[0,t]}$ in Definition\,\ref{defi1}.
Recently Levin and Wildon in \cite{LeWi} used iterated integrals and combinatorial
arguments involving the shuffle product (see \cite{Hoff}) to prove Theorem \ref{Levy
Sech}. Our method starts from the same iterated integrals. Hence, we use that the
integral in Definition\,\ref{defi1} can be written as
\begin{eqnarray}
\mathcal{A}_{[a,b)} &=&\dfrac{1}{2}\int_{a\,<x<\,y<\,b}\Bigl(
dX(x)dY(y)-dY(x) dX(y) \Bigr).  \label{classical}
\end{eqnarray}%
We may thus evaluate the moments as expectations of powers, using the so-called
\emph{sticky shuffle }\cite{Huds1} Hopf algebra. The multiplication in this algebra can
be used to express the product of two iterated It\^{o} stochastic integrals as a linear
combination of such iterated integrals. Since the expectation of an iterated integral
vanishes
unless each of the individual integrators is time, the \emph{recovery formula} \cite%
{Huds1,Bour}\ involving higher order Hopf algebra coproducts reduces the evaluation
of the moments to a combinatorial counting problem. Similar ideas were already used
in \cite{HuUwe} to calculate the moments in non-Fock quantum analogs of L\'{e}vy's area.

The sticky shuffle Hopf algebra is reviewed in Section\,2 and its use for reducing the
evaluation of moments to a counting problem is described in Section\,3. In Section\,4,
we apply combinatorial tools to show the main result. Finally, in the appendix, we
provide a simple lemmas about Euler numbers needed in our calculations.

\section{The sticky shuffle product Hopf algebra}

Let there be given an associative algebra $\mathcal{L}$ over $\mathbb{C}$.   The
corresponding vector space $\mathcal{T}(\mathcal{L})$ of tensors of all ranks over
$\mathcal{L}$ is defined as
\begin{equation}
\mathcal{T}(\mathcal{L})\,=\,\bigoplus\limits_{n=0}^{\infty
} 
\bigotimes\limits_{j=1}^{n}\mathcal{L}\,
.
\end{equation}
We denote by $\left( \alpha _{0},\alpha
_{1},\alpha _{2},...\right)$ the general element $\alpha=\alpha _{0}\oplus \alpha
_{1}\oplus \alpha _{2}\oplus \cdots $
of $\mathcal{T}(\mathcal{L})$, where only finitely many of the $\alpha _{m}$
are nonzero. For each $\alpha _{m}\in $ $\bigotimes_{j=1}^{m}%
\mathcal{L}$ the corresponding embedded element $\left( 0,0,...,\alpha
_{m},0,...\right) $ of $\mathcal{T}( \mathcal{L}) $ is denoted by
$\left\{ \alpha _{m}\right\} .$ In the following, we use the notational convention that, for arbitrary
elements $\alpha $ of $\mathcal{T}( \mathcal{L}) $ and $L$ of $%
\mathcal{L,}$ $\alpha \otimes L$ is the element of $\mathcal{T}(\mathcal{L}) $ for which $\left( \alpha \otimes L\right) _{0}=0$ and $%
\left( \alpha \otimes L\right) _{n}=\alpha _{n-1}\otimes L$ for $n\geq 1.$

The so-called \textit{sticky shuffle product Hopf algebra} over $\mathcal{L}$\ is formed
by equipping  $\mathcal{T}(\mathcal{L})$ with the operations of product, unit, coproduct
and counit defined as follows.

\begin{itemize}
\item The \textit{sticky shuffle product} of arbitrary elements of $\mathcal{T}(
    \mathcal{L}) $ is defined inductively by bilinear extension of the
rules%
\begin{eqnarray}
\{ 1_{\mathbb{C}}\}\left\{ L_{1}\otimes L_{2}\otimes
\cdots \otimes L_{m}\right\} &=&\left\{ L_{1}\otimes L_{2}\otimes \cdots
\otimes L_{m}\right\}\{ 1_{\mathbb{C}}\}\notag\\
&=&\left\{ L_{1}\otimes L_{2}\otimes \cdots \otimes L_{m}\right\} ,
\end{eqnarray}%
\begin{eqnarray}
&&\!\!\!\!\!\!\!\!\!\!\!\!\!\left\{ L_{1}\otimes L_{2}\otimes \cdots \otimes L_{m}\right\} \left\{
L_{m+1}\otimes L_{m+2}\otimes \cdots \otimes L_{m+n}\right\}  \notag \\
&=&\left( \left\{ L_{1}\otimes \cdots \otimes L_{m-1}\right\}
\left\{ L_{m+1}\otimes \cdots \otimes L_{m+n}\right\} \right)
\otimes L_{m}  \notag \\
&&+\left( \left\{ L_{1}\otimes \cdots \otimes L_{m}\right\}
\left\{ L_{m+1}\otimes \cdots \otimes L_{m+n-1}\right\}
\right) \otimes L_{m+n}  \label{sticky} \\
&&+\left( \left\{ L_{1}\otimes \cdots \otimes L_{m-1}\right\}
\left\{ L_{m+1}\otimes \cdots \otimes L_{m+n-1}\right\}
\right) \otimes L_{m}L_{m+n}.  \notag
\end{eqnarray}%

\item The \textit{unit element} for this product is $1_{\mathcal{T}(\mathcal{L}) }=\left(
    1_{\mathbb{C}},0,0,...\right) $.

\item The \textit{coproduct} $\Delta $ is the map from $\mathcal{T}(\mathcal{L}) $ to
    $\mathcal{T}(\mathcal{L}) \otimes \mathcal{T}(\mathcal{L}) $ defined by linear
    extension of the rules that $\Delta \left( 1_{\mathcal{T}(\mathcal{L}) }\right)
    =1_{\mathcal{T}(\mathcal{L}) }\otimes 1_{\mathcal{T}(\mathcal{L})
    }=1_{\mathcal{T}(\mathcal{L}) \otimes \mathcal{T}(\mathcal{L}) }$ and
\begin{eqnarray}
&&\!\!\!\!\!\!\!\!\!\!\!\!\!\Delta \left\{ L_{1}\otimes L_{2}\otimes \cdots \otimes L_{m}\right\}\notag \\
&=&1_{\mathcal{T}(\mathcal{L}) }\otimes \left\{ L_{1}\otimes
L_{2}\otimes \cdots \otimes L_{m}\right\} \\
&&+\sum\limits_{j=2}^{m}\left\{ L_{1}\otimes L_{2}\otimes \cdots \otimes
L_{j-1}\right\} \otimes \left\{ L_{j}\otimes L_{j+1}\otimes \cdots \otimes
L_{m}\right\}\notag \\
&&+\left\{ L_{1}\otimes L_{2}\otimes \cdots \otimes L_{m}\right\} \otimes 1_{%
\mathcal{T}(\mathcal{L}) }.\notag
\end{eqnarray}

\item The \textit{counit} $\varepsilon $ is the map from $\mathcal{T}(\mathcal{L}) $ to $\mathbb{C}$\ defined by linear extension of%
\begin{equation}
\varepsilon \left( 1_{\mathcal{T}(\mathcal{L}) }\right) =1_{%
\mathbb{C}}\ \text{ and }\ \varepsilon \left\{ L_{1}\otimes L_{2}\otimes \cdots
\otimes L_{m}\right\} =0\ \text{ for }\ m>0.
\end{equation}
\end{itemize}
\begin{remark}There is a useful alternative equivalent definition of the sticky shuffle product.
We can define the product $\gamma =\alpha \beta $ by%
\begin{equation}
\gamma _{N}\,=\sum_{A\cup B=\{1,2,...,N\}}\alpha _{\left\vert A\right\vert
}^{A}\beta _{\left\vert B\right\vert }^{B}.  \label{shuffle'}
\end{equation}%
Here the sum is now over the $3^{N}$ not necessarily disjoint ordered pairs $%
(A,B)$ whose union is $\{1,2,...,N\}$, and the notation is as follows; $%
\left\vert A\right\vert $ denotes the number of elements in the set $A$ so that $\alpha
_{\left\vert A\right\vert }$ denotes the homogeneous component of rank $\left\vert
A\right\vert $ of the tensor $\alpha =\left( \alpha _{0},\alpha _{1},\alpha _{2},...\right) ,$
and $\alpha _{\left\vert A\right\vert }^{A}$ indicates that this component is to be
regarded as occupying only those $\left\vert A\right\vert $ copies of $\mathcal{L}$
within $\bigotimes_{j=1}^{N}\mathcal{L}$ labelled by elements of the subset $A$ of
$\{1,2,...,N\}.$ Thus with $\beta _{\left\vert B\right\vert }^{B}$ defined analogously the
combination $\alpha _{\left\vert A\right\vert
}^{A}\beta _{\left\vert B\right\vert }^{B}$ is a well-defined element of $%
\bigotimes_{j=1}^{N}\mathcal{L}$. Here, if $A\cap B\neq \emptyset$, double occupancy
of a copy of $\mathcal{L}$ within $\bigotimes_{j=1}^{n}\mathcal{L}$ is reduced to single
occupancy by using the multiplication in the algebra $\mathcal{L}$ as a map from
$\mathcal{L\times L}$ to $\mathcal{L.}$
That (\ref{shuffle'}) is equivalent to (\ref{sticky}) is seen by noting that the three terms
on the right-hand side of (\ref{sticky}) correspond to the three mutually exclusive and
exhaustive possibilities that $N\in A\cap B^{c},$ $N\in A^{c}\cap B$ and $N\in A\cap B$
in (\ref{shuffle'}).
\end{remark}

The \emph{recovery formula }\cite{Bour}\emph{\ }expresses the homogeneous
components of an element $\alpha $ of $\mathcal{T}(\mathcal{L}) $ in terms of the
iterated coproduct $\Delta ^{(N)}\alpha$ by
\begin{equation}
\alpha _{N}=\left( \Delta ^{(N)}\alpha \right) _{(1,1,...,\overset{(N)}{1})}.
\label{recovery}
\end{equation}%
Here, $\Delta ^{(N)}$ is defined recursively by
\begin{equation}
\Delta ^{(2)}=\Delta\quad\text{and}\quad\Delta ^{(N)}=\left( \Delta \otimes \mbox{Id}_{\otimes ^{(N-2)}\left( \mathcal{T} \left(
\mathcal{L}\right) \right) }\right)\circ \Delta ^{(N-1)}\quad\text{for}\quad N>2\,.
\end{equation}
Hence, it is a map from $\mathcal{T}(\mathcal{L}) $ to the $N$th tensor power
\begin{equation}
\bigotimes\!^{(N)}\mathcal{T}(\mathcal{L})
\,=\,\bigotimes\!^{(N)}\bigoplus\limits_{n=0}^{\infty}\ \bigotimes\limits_{j=1}^{n}\mathcal{L}
\,=\bigoplus\limits_{n_{1},n_{2},...,n_{N}=0}^{\infty}\ \bigotimes_{r=1}^{N}\ \bigotimes\limits_{j_{r}=1}^{n_{r}}\mathcal{L}%
\end{equation}%
so that $\Delta ^{(N)}\alpha $ has multirank components $\alpha _{\left(
n_{1},n_{2},...,n_{N}\right) }$ of all orders. The recovery formula (\ref%
{recovery}) also holds when $N=0$ and $N=1$ if we define $\Delta ^{(0)}$ and
$\Delta ^{(1)}$ to be the counit $\varepsilon $ and the identity map $1_{%
\mathcal{T}(\mathcal{L}) }$ respectively.

Note that\ $\Delta $ is multiplicative, $\Delta \left( \alpha \beta \right) =\Delta \left( \alpha
\right) \Delta \left( \beta \right)$, where the product on the tensor square
$\mathcal{T}(\mathcal{L})\otimes\mathcal{T}(\mathcal{L})$ is defined by linear
extension of the rule
\begin{equation}\label{eq.power}
(a\otimes a')(b\otimes b')\,=\,ab\otimes a'b'\,.
\end{equation}

\section{Moments and sticky shuffles}

We now describe the connection between sticky shuffle products and iterated
stochastic integrals. We begin with the well-known fact that, for the
one-dimensional Brownian motion $X$ and for $a\leq b,$%
\begin{equation}
\left( X(b)-X(a\right) )^{2}\,=\,2\!\int_{a\leq x<b}\!\!\!\!\!\!\!\!\left( X(x)-X(a)
\right) dX(x)\ +\,\int_{a\leq x<b}\!\!\!\!\!\!\!\!dT(x),  \label{k}
\end{equation}%
where $T(x)=x$ is time. We introduce the \emph{It\^{o} algebra} $\mathcal{L=}%
\mathbb{C}\left\langle dX,dT\right\rangle $ of complex linear combinations of the basic
differentials $dX$ and $dT,$ which are multiplied according to the table
\begin{equation}
\begin{tabular}{c|c}
& $%
\begin{array}{cc}
dX & dT%
\end{array}%
$ \\ \hline
$%
\begin{array}{c}
dX \\
dT%
\end{array}%
$ & $%
\begin{array}{cc}
dT & 0 \\
0 & 0%
\end{array}%
$%
\end{tabular}%
,  \label{table}
\end{equation}%
together with the corresponding sticky shuffle Hopf algebra $\mathcal{T}%
(\mathcal{L}) .$ For each pair of real numbers $a<b,$ we introduce a map $J_{a}^{b}$
from $\mathcal{T}(\mathcal{L}) $ to complex-valued random variables on the
probability space of the Brownian
motion $X$ by linear extension of the rule that, for arbitrary $%
dL_{1},dL_{2},\cdots dL_{m}\in \left\{ dX,dT\right\} $
\begin{eqnarray}
&&\!\!\!\!\!\!\!\!\!\!\!\!\!\!J_{a}^{b}\left\{ dL_{1}\otimes dL_{2}\otimes \dotsm \otimes dL_{m}\right\}\notag
\\
&=&\!\!\!\!\int_{a\leq x_{1}<x_{2}<\dotsb
<x_{m}<b}dL_{1}(x_{1})\,dL_{2}(x_{2})\,dL_{3}(x_{3})\dotsm\,dL_{m}(x_{m})\\
&=&\!\!\!\!\int_{a}^{b}\!\!\dotsm\int_{a}^{x_{4}}\!\!\int_{a}^{x_{3}}\!\!\int_{a}^{x_{2}}
\!\!dL_{1}(x_{1})\,dL_{2}(x_{2})\,dL_{3}(x_{3})\dotsm\,\notag
dL_{m}(x_{m}) .
\end{eqnarray}%
By convention $J_{a}^{b}$ maps the unit element of the algebra $\mathcal{T}%
(\mathcal{L}) $ to the unit random variable identically equal to 1.

Then (\ref{k}) can be restated as follows,
 \begin{equation}
J_{a}^{b}\left(\left\{
dX\right\} \right)J_{a}^{b}\left(\left\{
dX\right\} \right)=J_{a}^{b}\left(\left\{
dX\right\} \left\{
dX\right\}\right),
\end{equation}
using the fact that $\left\{ dX\right\} ^{2}=2\left\{ dX\otimes dX\right\} +\left\{ dT\right\} .$

The following more general Theorem is probably known to many probabilists.

\begin{theorem}\label{general}
For arbitrary $\alpha $ and $\beta $ in $\mathcal{T}(\mathcal{L})$,
\begin{equation*} \label{j}
J_{a}^{b}(\alpha )J_{a}^{b}(\beta )=J_{a}^{b}(\alpha \beta )\,.
\end{equation*}
\end{theorem}
\begin{proof}
By bilinearity it is sufficient to consider the case when
\begin{equation}
\alpha =\left\{ dL_{1}\otimes dL_{2}\otimes \cdots \otimes dL_{m}\right\}
,\ \beta =\left\{ dL_{m+1}\otimes dL_{m+2}\otimes \cdots \otimes
dL_{m+n}\right\}
\end{equation}%
for $dL_{1},dL_{2},\cdots ,dL_{m+n}\in \left\{ dX,dT\right\} .$ In this case Theorem
\ref{general} follows, using the inductive definition (\ref{sticky}) for the
sticky shuffle product, from the product form of It\^{o}'s formula,%
\begin{equation}
d\left( \xi \eta \right) =\left( d\xi \right) \eta +\xi d\eta +\left( d\xi
\right) d\eta   \label{ito}
\end{equation}%
where stochastic differentials of the form $d\xi =FdX+GdT,$ with stochastically
integrable processes $F$ and $G$, are multiplied using the table (\ref%
{table}).\bigskip
\end{proof}

For planar Brownian motion $B=\left( X,Y\right) $ the Ito table (\ref{table}%
)~becomes%
\begin{equation}
\begin{tabular}{c|c}
& $%
\begin{array}{cc}
\begin{array}{cc}
dX & dY%
\end{array}
& dT%
\end{array}%
$ \\ \hline
$%
\begin{array}{c}
\begin{array}{c}
dX \\
dY%
\end{array}
\\
dT%
\end{array}%
$ & $%
\begin{array}{cc}
\begin{array}{cc}
dT & 0  \\
0 & dT%
\end{array}
&
\begin{array}{c}
0 \\
0%
\end{array}
\\
\begin{array}{cc}
0\ \ & 0%
\end{array}
& 0%
\end{array}%
$%
\end{tabular}%
.  \label{table'}
\end{equation}%

\begin{corollary}
Theorem \ref{general} holds when $\mathcal{L}$ is the algebra defined by the
multiplication table (\ref{table'}).
\end{corollary}


Basic for our calculations is the next theorem. It follows from the fact that expectations
of stochastic integrals against either $dX$ or $dY$ as integrators are zero.

\begin{theorem}\label{sz.LT}
For arbitrary $n\in \mathbb{N,}$ $a<b\in \mathbb{R}$ and basis elements $dL_{1}$,
$dL_{2}$, \dots, $dL_{n}\in\{dX,dY,dT\}$,
\begin{equation*}
\mathbb{E}\!\left[ J_{a}^{b}\left\{ dL_{1}\otimes dL_{2}\otimes ...\otimes
dL_{n}\right\} \right] \,=\,0
\end{equation*}%
unless%
\begin{equation*}
dL_{1}=dL_{2}=\cdots =dL_{n}=dT.
\end{equation*}
\end{theorem}



In view of (\ref{classical})
\begin{equation}\label{xxyxxy}
\mathcal{A}_{[a,b)}\,=\,\dfrac{1}{2}J_{a}^{b}(dX\otimes dY-dY\otimes dX)\,.
\end{equation}

Now consider the moments sequence of classical L\'{e}vy area in terms of the basis
$\left( dX,dY,dT\right) $, i.e., Eqn. (\ref{xxyxxy}). In view of Theorem \ref{general}
\begin{eqnarray}
\left[ \mathcal{A}_{[a,b)}\right] ^{n}
&=&{\dfrac{1}{2^n}}\left( J_{a}^{b}( dX\otimes dY-dY\otimes dX)  \right) ^{n} \notag \\
&=&{\dfrac{1}{2^n}}J_{a}^{b}\left( \{ dX\otimes dY-dY\otimes dX\} ^{n}\right)
\end{eqnarray}

The $n$th sticky shuffle power $\left\{ dX\otimes dY-dY\otimes dX\right\} ^{n}$ will
consist of non-sticky shuffle products of rank $2n$ together with terms of lower ranks
$n,n+1,...,2n-1$,$~$all of which except the rank $n$ term will contain one or more
copies of $dX$ and $dY$, and will thus not contribute to the expectation in view of
Theorem\,\ref{sz.LT}. The term of rank $n$ will be a multiple
of $dT\otimes dT\cdots \otimes \overset{(n)}{dT}.$ Thus we can write%
\begin{equation}
\left\{ dX\otimes dY-dY\otimes dX\right\} ^{n}
\,=\,w_{n}\bigl\{ dT\otimes dT\cdots \otimes \overset{%
(n)}{dT}\bigr\} +\,\text{terms of rank}>n.  \label{rank}
\end{equation}%
for some coefficient $w_{n}$. The corresponding moment is given by
\begin{eqnarray}
\mathbb{E}\!\left[ \mathcal{A}_{[a,b)}\right]
^{n} &=&\dfrac{w_{n}}{2^{n}}\,\mathbb{E}\!\left[ J_{a}^{b}\bigl(
\{ dT\otimes dT\cdots \otimes \overset{(n)}{dT}\} \bigr) \right]
\notag \\
&=&\dfrac{w_{n}}{2^{n}}\int_{a\leq x_{1}<x_{2}<\cdots
<x_{n}<b}dx_{1}\,dx_{2}\dotsm\,dx_{n}  \notag \\
&=&\dfrac{w_{n}\left( b-a\right) ^{n}}{2^{n}n!}.
\label{moment}
\end{eqnarray}

By the recovery formula (\ref{recovery}) and the multiplicativity of the $n$%
th order coproduct $\Delta ^{\left( n\right) },$%
\begin{eqnarray}
&&\!\!\!\!\!\!\!\!\!\!\!\!w_{n}dT\otimes dT\cdots \otimes \overset{(n)}{dT} \notag\\
&=&\left\{\left\{dX\otimes dY-dY\otimes dX\right\} ^{n}\right\}_n \notag\\
&=&\left( \Delta ^{(n)}\bigl( \left\{ dX\otimes dY-dY\otimes dX\right\}
_{{}}^{n}\bigr) \right) _{\!\!(1,1,...,\overset{\left( n\right) }{1})} \notag\\
&=&\left( \Bigl( \Delta ^{(n)}\bigl( \left\{ dX\otimes dY-dY\otimes dX\right\} \bigr)
\Bigr) ^{n}\right) _{\!\!(1,1,...,\overset{\left( n\right) }{1})}.
\end{eqnarray}%
Now%
\begin{eqnarray}
&&\!\!\!\!\!\!\!\!\!\!\!\!\Delta ^{(n)}\bigl( \left\{ dX\otimes dY-dY\otimes dX\right\} \bigr)  \notag\\
&=&\sum_{1\leq j\leq n}1_{\mathcal{T}(\mathcal{L}) }\otimes\cdots \otimes\overset{(j)}{\left\{dX\otimes dY-dY\otimes dX\right\}}
\otimes\cdots \otimes 1_{\mathcal{T}(\mathcal{L}) }\notag \\
&&+\sum_{1\leq j\,<k\leq n}\biggl( 1_{\mathcal{T}( \mathcal{L}) }\otimes \cdots \otimes \overset{\left( j\right) }{\left\{ dX\right\} }\otimes \cdots \otimes \overset{\left(
k\right) }{\left\{ dY\right\} }%
\otimes \cdots \otimes \overset{\left( n\right) }{1_{\mathcal{T}(\mathcal{L}) }} \notag \\
&& -\ 1_{\mathcal{T}(\mathcal{L}) }\otimes \cdots \otimes
\overset{\left( j\right) }{\left\{ dY\right\} }\otimes \cdots \otimes \overset{\left( k\right) }{\left\{ d{X}%
\right\} }\otimes \cdots \otimes \overset{\left(
n\right) }{1_{\mathcal{T}(\mathcal{L}) }}\biggr)
\end{eqnarray}
The first term of this sum, being of rank $2,$ cannot contribute to the component of joint rank $(1,1,...,%
\overset{\left( n\right) }{1})$ of the $n$th power of $\Delta ^{(n)}\left( \left\{ dX\otimes
dY-dY\otimes dX\right\} \right)$, where product in the n\textit{th} tensor power
$\bigotimes\!^{(N)}\mathcal{T}(\mathcal{L})$ is defined exactly as in the case $n=2$ in
\eqref{eq.power}. Thus
\begin{eqnarray}\label{eq.111}
&&\!\!\!\!\!\!\!\!\!\!\!\!w_{n}dT\otimes dT\cdots \otimes \overset{(n)}{dT} \notag\\
&=&\left( \left( \Delta ^{(n)}\bigl( \left\{dX\otimes dY-dY\otimes dX\right\} \bigr)
\right) ^{\!n\,}\right) _{\!\!(1,1,...,\overset{\left( n\right) }{1})} \\
&=&\left( \left( \sum_{1\leq j\,<k\leq n}\biggl( 1_{\mathcal{T}(
\mathcal{L}) }\otimes \cdots \otimes \overset{(j)}{%
\left\{ dX\right\} }\otimes \cdots \otimes
\overset{(k)}{\left\{ dY\right\} }\otimes \cdots \otimes \overset{(n)}{1_{\mathcal{T}%
(\mathcal{L})}}\biggr. \right. \right. \notag \\
&&\left. \left. \biggl.-\ 1_{\mathcal{T}(\mathcal{L}) }\otimes
\cdots \otimes \overset{(j)}{\left\{ dY\right\} }\otimes \cdots \otimes \overset{(k)}{%
\left\{ dX\right\} }\otimes \cdots \otimes
\overset{(n)}{1_{\mathcal{T}(\mathcal{L}) }}\biggr)
\right)^{\!\!\!n\,}\right) _{\!\!\!(1,1,...,\overset{\left( n\right) }{1})}\notag
\end{eqnarray}%
This calculation of $w_{n}$ can be finished using some combinatorics. We do that in
the following section.

\section{The moments of L\'{e}vy's area}

To evaluate the moments $\mathbb{E}\!\left[ \mathcal{A}_{[a,b)}\right] ^{n}$\!, we need
to calculate the number $w_{n}$, as explained in (\ref{moment}). By \eqref{eq.111}, we
have
\begin{equation}\label{eq.start}
w_{n}\,dT\otimes dT\cdots \otimes \overset{(n)}{dT}
\,\,=\,\left(\left(\sum_{h\neq k}\sn(h,k)R_{h,k}\right)^{\!\!\!n\,}\right)_{\!\!\!(1,1,...,\overset{\left( n\right) }{1})}
\end{equation}
with
\begin{equation}
R_{h,k}\,:=\,1\otimes \dots \otimes 1\otimes \overset{(h)}{\{dX\}}\otimes 1\otimes \dots \otimes 1\otimes \overset{(k)}{\{dY\}}\otimes 1\otimes \dots \otimes 1  \label{40}
\end{equation}%
and 
\begin{equation}
\sn(h,k)\,:=\,%
\begin{cases}
+1 & \text{if $h< k$}, \\
-1 & \text{if $h>k$}.%
\end{cases}
\label{41}
\end{equation}%
The $n$th power in \eqref{eq.start} is based on the sticky shuffle product in
$\mathcal{T}(\mathcal{L})$ and its extension to the n\textit{th} tensor power
$\bigotimes\!^{(n)}\mathcal{T}(\mathcal{L})$, as described in \eqref{eq.power} for
$n=2$.

If we set $e:=(h,k)$, then we may also write $R_{e}$ for $R_{h,k}$ and $\sn(e)$
for $\sn(h,k)$. Using distributivity, 
this yields
\begin{equation}
w_{n}\,dT\otimes dT\cdots \otimes \overset{(n)}{dT}
\,=\,\sum\ \left(\prod_{\ell=1}^n\sn(e_\ell)\right)\!\!\left(\prod_{\ell=1}^nR_{e_\ell}\right)_{\!\!\!(1,1,...,\overset{\left( n\right) }{1})}
\,,\label{eq.sum}
\end{equation}
where the sum runs over all $n$-tuples $(e_1,e_2,\dotsc, e_n)$ of pairs $(h,k)$ with
$h\neq k$. We may imagine each pair $e_\ell=(h_\ell,k_\ell)$ as a directed edge, an
\emph{arc}, from $h_{\ell }$ to $k_{\ell }$. Each $n$-tuples $(e_1,e_2,\dotsc, e_n)$ is
then a directed labeled multigraph, we say a \emph{digraph}, on the vertex set
$V:=\{1,2,\dotsc,n\}$. We have to see what the individual arcs $e_{\ell}$ of a digraph
$(e_1,e_2,\dotsc, e_n)$ contribute to its corresponding summand
$\pm\prod_{\ell=1}^nR_{e_\ell}$ inside the sum \eqref{eq.sum}.
For example, in the case $n=4$, the two arcs $e_{1}=(1,2)$ and $e_{2}=(3,2)$
contribute
\begin{equation}\label{42}
\begin{split}
\sn(e_1)&\sn(e_2)\bigl(R_{e_1}R_{e_2}\bigr)_{\!(1,1,1,1)}\\
&\,=\,-\bigl((\{dX\}\otimes \{dY\}\otimes 1\otimes 1)(1\otimes\{dY\}\otimes \{dX\}\otimes 1)\bigr)_{\!(1,1,1,1)} \\
&\,=\,-\bigl(\{dX\}1\bigr)_{\!(1)}\otimes\bigl(\{dY\}\{dY\}\bigr)_{\!(1)}\otimes\bigl(1\{dX\}\bigr)_{\!(1)}\otimes\bigl(1\cdot1\bigr)_{\!(1)}\\
&\,=\,-\,dX\otimes dT\otimes dX\otimes 1\,,
\end{split}
\end{equation}
where we basically ignored $e_3$ and $e_4$ (and the corresponding $R_{e_3}$,
$R_{e_4}$, $\sn(e_3)$ and $\sn(e_4)$) to illustrate how the product operates.

In order to calculate the coefficient $w_{n}$ of $dT\otimes dT\otimes \dotsm \otimes
dT$ in \eqref{eq.sum}, we need to retain only those summands
$\pm\prod_{\ell=1}^nR_{e_\ell}$ that contribute a scalar multiple of $dT\otimes
dT\otimes \dotsm \otimes dT$. We may discard other summands. Hence, we do not
have to sum over all digraphs $(e_1,e_2,\dotsc, e_n)$. To see which ones we have to
retain, let us assume that $(e_1,e_2,\dotsc, e_n)$ yields a multiple of $dT\otimes
dT\otimes \dotsm \otimes dT$ in \eqref{eq.sum}. Since the $n$ copies of $dX$ and $n$
copies of $dY$ in the unexpanded product $\prod_{\ell=1}^nR_{e_\ell}$ must yield $n$
copies of $dT$\!, one in each possible position, each vertex of the digraph
$(e_1,e_2,\dotsc, e_n)$ must have either exactly two incoming and no outgoing arcs
(corresponding to a $(dY)^2$\,) or exactly two outgoing and no incoming arcs
(corresponding to a $(dX)^2$\,).
This shows that, inevitable, $(e_1,e_2,\dotsc, e_n)$ must consist of disjoint
alternatingly oriented cycles that cover $V$\!, cycles whose arcs go ``forward -
backward - forward - backward - \dots''. Every such digraph $(e_1,e_2,\dotsc, e_n)$
has necessarily an even number of vertices, $n=2m$, and contributes either $+1$ or
$-1$ to $w_{2m}$. In particular, this means that for odd $n$ we do not obtain any term
$dT\otimes dT\otimes \dotsm \otimes dT$ in \eqref{eq.sum}, i.e.\ $w_n=0$ for odd $n$.

Before we do the counting in the case $n=2m$, we transform the alternatingly oriented
labeled digraphs $(e_1,e_2,\dotsc, e_{2m})$ into cyclically oriented digraphs $D$, and,
eventually, into permutations of a certain kind. Turning around each second arc in each
cycle, we get cyclicly oriented cycles (like cyclic one way roads). These disjoint cycles
still cover $V$ and have even length, as they arose from alternatingly oriented cycles.
Conversely, we can always go back to alternatingly oriented cycles by flipping each
second arc. To be precise, each (labeled) even cycle has two alternating orientations
but also two cyclic orientations. Hence, there are several ways to match our labeled
alternatingly oriented digraphs $(e_1,e_2,\dotsc, e_{2m})$ and the new cyclically
oriented digraphs. However, they are all valid. We only need to know that there exist a
bijection, a bijection under which every image differs from its pre-image
$(e_1,e_2,\dotsc, e_{2m})$ in exactly $m$ edges. This will then result in an additional
factor of $(-1)^m$ in our calculations. Moreover, at this point, our cyclically oriented
digraphs do not contain multiple arcs, so that we may forget the labels $1,2,\dotsc,2m$
of the arcs $e_1,e_2,\dotsc, e_{2m}$. If a digraph $D$ with $2m$ many unlabeled arcs
has no multiple (no indistinguishable) arcs, then it corresponds to exactly $(2m)!$ many
labeled digraphs, yielding a factor of $(2m)!$ in our sum. Very careful readers might be
astonished that we could not drop the edge labels earlier in this way. We invite them to
investigate the case $m=1$ to see why.

Eventually, we can now turn towards permutations in the symmetric group $\cS_{2m}$
as representatives for the diagraphs $D$ and the corresponding summands
$\pm\prod_{\ell=1}^nR_{e_\ell}$ inside the sum of \eqref{eq.sum}. We may view each
arc $(h,k)$ in any cyclically oriented unlabeled digraph $D$ as the assignment of a
function value, $h\mapsto k=:\cs(h)$, and obtain a permutation $\cs=\cs_{D}$ on
$V=\{1,2,\dotsc,2m\}$. In our case, the cycles of $\cs$ have even length. We denote
with $\cD_{2m}\sb\cS_{2m}$ the set of all permutations of this kind. Putting all this
together, and keeping track of the signs, we see that
\begin{equation}
w_{2m}\,=\,(-1)^m(2m)!\sum_{\cs\in\cD_{2m}}\sn(\cs)\,,
\end{equation}
where
\begin{equation}
\sn(\cs)\,:=\,\prod_{j=1}^{2m}\sn(j,\cs(j))\,.
\end{equation}

To determine this sum, we cancel off some summands $\cs$ with opposite signs
$\sn(\cs)$. We call a point $h\in V$ a \emph{transit}
of $\cs$ if
\begin{equation}
\text{either}\quad\cs^{-1}(h)<h<\cs(h)\quad\text{or}\quad\cs^{-1}(h)>h>\cs(h)\ .
\end{equation}
Let $\cD_{2m}'$ bet the set of all $\cs\in\cD_{2m}$ which have at least one transit. We
show that the elements of $\cD_{2m}'$ cancel out completely and can be ignored in our
sum. Obviously, every $\cs\in\cD_{2m}'$ contains a unique smallest transit $h$, and we
obtain a permutation $\cs'$ of $V\!\sm\!\{h\}$ by replacing the chain of assignments
$\cs^{-1}(h)\mto h\mto\cs(h)$ with the shorter chain $\cs^{-1}(h)\lmto\cs(h)$. The new
permutation $\cs'$ has a unique odd cycle
\begin{equation}
j_1\mto j_2\mto\dotsb\mto\cs^{-1}(h)\lmto\cs(h)\mto\dotsm\mto j_{2k-1}\mto j_1.
\end{equation}
If we walk once around this cycle and observe the indices $j_1,j_2,j_3,\dotsc$ as a kind
of altitude, then we will cross the altitude $h$ as many times upwards, from below $h$
to above $h$, as downwards, from above $h$ to below $h$. Hence, there is an even
number of ways to reinsert $h$ as transit into that odd cycle, see Fig.\,\ref{fig.1}.
\begin{figure}[t]
\begin{center}\label{fig.1}
\vspace{-1em}
\includegraphics[
scale = .13
]{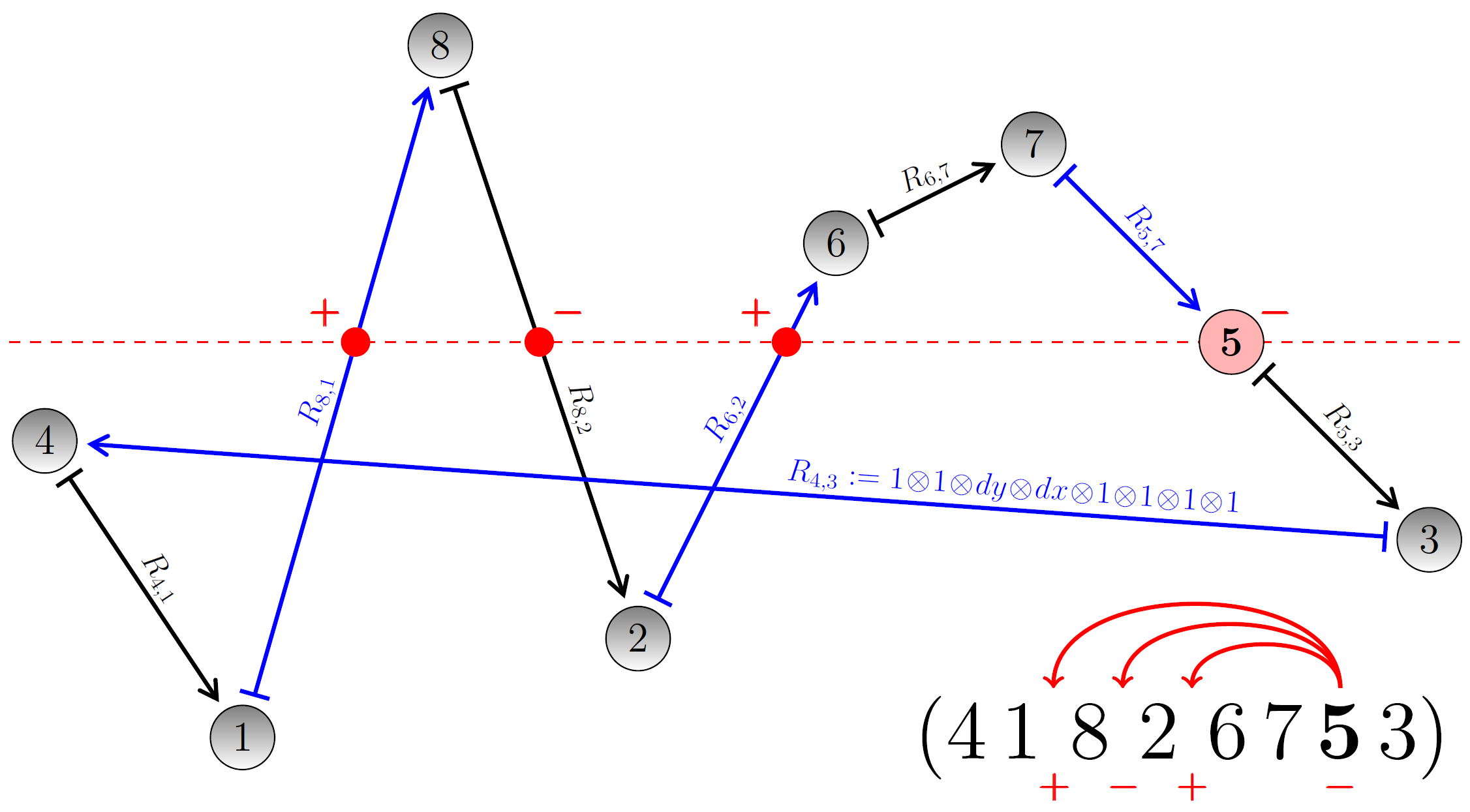}
\vspace{-1em}
\caption{The cyclic permutation $\cs=(4,1,8,2,6,7,5,3)$ with smallest transit $h=5$.}
\end{center}
\end{figure}
One half of the permutations that we obtain will have positive sign, one half negative
sign. Removal and reinsertion of a smallest transit yields an equivalence relation
$\sim$ on $\cD_{2m}'$, $\cs\sim\crr$ if and only if $\cs'=\crr'$. The corresponding
equivalence classes form a partition of $\cD_{2m}'$, and each of them cancels off
nicely. So, we only have to sum over $\cD_{2m}\!\sm\!\cD_{2m}'$, that is, over all
permutations $\cs\in \cS_{2m}$ with
\begin{equation}
\text{either}\quad\cs^{-1}(j)<j>\cs(j)\quad\text{or}\quad\cs^{-1}(j)>j<\cs(j)
\end{equation}
for all $j\in V:=\{1,2,\dotsc,2m\}$. We call this kind of permutations \emph{forth-back
permutations}. Their number is the so-called \emph{Euler zigzag number} $A_{2m}$,
i.e.\ $|\cD_{2m}\!\sm\!\cD_{2m}'|=A_{2m}$, as shown in 
Lemma\,\ref{lem.fb} in the appendix. Since all forth-back permutations have sign
$(-1)^m$, we get
\begin{equation}
w_{2m}\,=\,(2m)!\,A_{2m}.
\end{equation}

From this and Equation\,(\ref{moment}) we finally arrive at L\'{e}vy's classical result
\eqref{Levy Mom}:

\begin{theorem}
The nonzero moments of the L\'{e}vy area $\mathcal{A}_{[a,b)}$ are%
\begin{equation*}
\mathbb{E}\!\left[ \mathcal{A}_{[a,b)}\right]^{2m}\,=\,\,\left(\frac{ b-a}{2}\right)^{\!2m}\!A_{2m}\,.
\end{equation*}
\end{theorem}

%

\section{Appendix about Euler Numbers}

In this section, we present a 
simple lemmas about Euler numbers. It is of sufficient general nature to be of potential
interest elsewhere. Many similar results and basics can be found in \cite{pe} and
\cite{st}.

A permutation $\mathfrak{s}$ in the symmetric group $\cS_{n}$ is a \emph{zigzag
permutation} (misleadingly also called alternating permutation)
if $\mathfrak{s}(1)>\mathfrak{s}(2)<\mathfrak{s}(3)>\mathfrak{s}(4)<\dotsb $%
. In other words, $\mathfrak{s}$ is zigzag if $\mathfrak{s}(1)>\mathfrak{s}(2)$ and
\begin{equation}
\text{either}\quad\mathfrak{s}(j-1)<\mathfrak{s}(j)>\mathfrak{s}(j+1)\quad\text{or}\quad\mathfrak{s}(j-1)>\mathfrak{s}(j)<\mathfrak{s}(j+1)  \label{1}
\end{equation}%
for all $j\in \{2,3\dotsc ,n-1\}$.
If we have the initial condition $\mathfrak{s}(1)<\mathfrak{s}(2)$, instead of $\mathfrak{s}(1)>%
\mathfrak{s}(2)$, we may call $\mathfrak{s}$ \emph{zagzig}. The number of all zigzag
permutations in $\cS_{n}$ is called the \emph{Euler zigzag number} $A_{n}$. These
numbers occur in many places, for example, as the coefficients of $\frac{z^{2n}}{(2n)!}$
in the Maclaurin series of $\sec
(z)+\tan (z)$. In this paper, we met them as the number of \emph{%
forth-back permutations}, as we called them. These are the permutations $%
\mathfrak{s}\in \cS_{n}$ with
\begin{equation}
\text{either}\quad\mathfrak{s}^{-1}(j)<j>\mathfrak{s}(j)\quad\text{or}\quad\mathfrak{s}^{-1}(j)>j<\mathfrak{s}(j)  \label{2}
\end{equation}%
for all $j\in \{1,2,\dotsc ,n\}$. Since no forth-back permutation can contain a cycle of odd
length, $n$ must be even for there to exist forth-back permutations, say $n=2m>0$. In
that case, we actually have the following lemma:

\begin{lemma}
\label{lem.fb} The number of forth-back permutations in $\cS_{2m}$ is the Euler zigzag
number $A_{2m}$.
\end{lemma}

\begin{proof}
A bijection between the forth-back permutations $\mathfrak{s}$ and the zigzag
permutations in $\cS_{2m}$ is obtained by applying the so-called \emph{transformation
fundamentale} \cite{FoSc}. To perform this transformation, we write $\mathfrak{s}$ in
cycle notation
\begin{eqnarray}
\mathfrak{s} &=&(\mathfrak{s}_{1},\mathfrak{s}_{2},\dotsc ,\mathfrak{s}%
_{\ell _{2}-1})(\mathfrak{s}_{\ell _{2}},\mathfrak{s}_{\ell _{2}+1},\dotsc ,%
\mathfrak{s}_{\ell _{3}-1})(\mathfrak{s}_{\ell _{3}},\mathfrak{s}_{\ell
_{3}+1},\dotsc ,\mathfrak{s}_{\ell _{4}-1})\dotsm  \notag \\
&&(\mathfrak{s}_{\ell _{m}},\mathfrak{s}_{\ell _{m}+1},\dotsc ,\mathfrak{s}%
_{2m})\,.  \label{3}
\end{eqnarray}%
This representation and the numbers \ $\mathfrak{s}_{j}$ are uniquely determined if we
require that the first entry of every cycle is bigger than
all other entries in that cycle, and also that $\mathfrak{s}_{1}<\mathfrak{s}%
_{\ell _{2}}<\mathfrak{s}_{\ell _{3}}<\dotsb <\mathfrak{s}_{\ell _{m}}$. The new
permutation $\mathfrak{\bar{s}}$ is then obtained by forgetting brackets and setting
$\mathfrak{\bar{s}}(j):=\mathfrak{s}_{j}$. We just have to see that this actually yields a
bijection $\mathfrak{s}\mapsto \mathfrak{\bar{s}} $ between forth-back and zigzag
permutations. To do this we procede as follows.

Assume first that $\mathfrak{s}$ is forth-back. Then all cycles necessarily have even
length and the permutation $\mathfrak{\bar{s}}$ is obviously
zigzag, $\mathfrak{s}_{1}>\mathfrak{s}_{2}<\mathfrak{s}_{3}>\mathfrak{s}%
_{4}<\dotsb >\mathfrak{s}_{2m}$. Conversely, let us show that every zigzag
permutation $\bar{\mathfrak{s}}$ has a unique pre-image $\mathfrak{s}$, and
that that pre-image is forth-back.  To construct a pre-image $\mathfrak{s}$ of $%
\bar{\mathfrak{s}}$, we only need to find suitable numbers $\ell _{j}$,
which indicate where we have to insert brackets into the sequence $(%
\mathfrak{s}_{1},\mathfrak{s}_{2},\dotsc ,\mathfrak{s}_{2m}):=(\mathfrak{s}%
(1),\mathfrak{s}(2),\dotsc ,\mathfrak{s}(2m))$ to actually get a pre-image. However, if
we have already found $\ell _{2},\ell _{3},\dotsc ,\ell _{j}$,
then $\ell _{j+1}$ is necessarily the first index $x$ with $\mathfrak{s}_{x}>%
\mathfrak{s}_{\ell _{j}}$. Using this, we can construct a pre-image $%
\mathfrak{s}$ of $\bar{\mathfrak{s}}$ in $\cS_{2m}$, and it is uniquely determined.
Moreover, if $\bar{\mathfrak{s}}$ is zigzag then this construction ensures that
$\mathfrak{s}_{1}$ and the $\mathfrak{s}_{\ell _{j}}$ are peaks and their neighbors and
$\mathfrak{s}_{2m}$ are valleys.
Since also $\mathfrak{s}_{1}>\mathfrak{s}_{\ell _{2}-1}$, $\mathfrak{s}%
_{\ell _{2}}>\mathfrak{s}_{\ell _{3}-1}$, \dots , $\mathfrak{s}_{\ell _{m}}>%
\mathfrak{s}_{2m}$, insertion of brackets before the peaks $\ell _{j}$ yields forth-back
cycles in $\mathfrak{s}$.

With the bijection established, it is now clear that there are as many
forth-back permutations as there are zigzag permutations in $\cS%
_{2m} $. This number is the Euler zigzag number $A_{2m}$.
\end{proof}

The number of forth-back permutations with just one cycle is given by the following
lemma, which we present here as we think that it can be helpful in future research. If
$\cC_n$ denotes the subset of cyclic permutations in $\cS_n$, we have the following:

\begin{lemma}
\label{lem.cfb} The number of forth-back permutations in $\cC_{2m}$ is $A_{2m-1}$.
\end{lemma}

\begin{proof}
The cycle notation $\mathfrak{s}=(\mathfrak{s}_{1},\mathfrak{s}_{2},\dotsc ,%
\mathfrak{s}_{2m})$ of cyclic permutations $\mathfrak{s}\in \cS_{2m}$ is not uniquely
determined, as one may rotate the entries cyclically. It becomes uniquely determined if
we require that $\mathfrak{s}_{2m}=2m$. In
this case, removal of the last entry yields a sequence $(\mathfrak{s}_{1},%
\mathfrak{s}_{2},\dotsc ,\mathfrak{s}_{2m-1})$ that is zagzig (with $%
\mathfrak{s}_{1}<\mathfrak{s}_{2}$ as $\mathfrak{s}_{2m}$ was the biggest
entry of $\mathfrak{s}$). If we define $\bar{\mathfrak{s}}\in \cS%
_{2m-1}$ by setting $\bar{\mathfrak{s}}(j):=\mathfrak{s}_{j}$, for $%
j=1,2,\dotsc ,2m-1$, we obtain a bijection $\mathfrak{s}\mapsto \bar{%
\mathfrak{s}}$ from the cyclic forth-back permutations in $\cS_{2m}$ to the zagzig
permutations in $\cS_{2m-1}$. Indeed, every zagzig
permutation $\bar{\mathfrak{s}}$ in $\cS_{2m-1}$ has the cycle $%
\mathfrak{s}:=(\bar{\mathfrak{s}}(1),\bar{\mathfrak{s}}(2),\dotsc ,\bar{%
\mathfrak{s}}(2m-1),2m)$ as unique pre-image. The existence of this
bijection shows that the number of cyclic forth-back permutations in $%
\cS_{2m}$ is equal to the number of zagzig permutations in $\mathcal{%
S}_{2m-1}$, which is $A_{2m-1}$, as for zigzag permutations.
\end{proof}

\end{document}